\documentclass{amsart}
\usepackage{amssymb,amsmath,latexsym,mathrsfs}
\usepackage{times}

\numberwithin{equation}{section}

\newtheorem{thm}{Theorem}[section]

\newtheorem{cor}[thm]{Corollary}
\newtheorem{prop}[thm]{Proposition}
\newtheorem{lemma}[thm]{Lemma}

\newcommand{\del}{\backslash}
\newcommand{\cl}{\hbox{\rm cl}}
\newcommand{\mcZ}{\mathcal{Z}}
\newcommand{\mcA}{\mathcal{A}}
\newcommand{\mcF}{\mathcal{F}}
\newcommand{\mcC}{\mathcal{C}}
\newcommand{\mcE}{\mathcal{E}}
\newcommand{\mcH}{\mathcal{H}}
\newcommand{\mbH}{\mathscr{H}}
\newcommand{\join}{\lor}
\newcommand{\meet}{\land}

\title[A Construction of Intertwines]{A Construction of Infinite Sets
  of Intertwines \\ for Pairs of Matroids}
                   
\date{\today} 

\author[J.~Bonin]{Joseph E.~Bonin} 

\address{Department of Mathematics\\ The George Washington
  University\\ Washington, D.C. 20052}

\email{jbonin@gwu.edu} 



\begin{document}

\begin{abstract}
  An intertwine of a pair of matroids is a matroid such that it, but
  none of its proper minors, has minors that are isomorphic to each
  matroid in the pair.  For pairs for which neither matroid can be
  obtained, up to isomorphism, from the other by taking free
  extensions, free coextensions, and minors, we construct a family of
  rank-$k$ intertwines for each sufficiently large integer $k$.  We
  also treat some properties of these intertwines.
\end{abstract}

\maketitle

\section{Introduction}

If the classes $\mcC_1$ and $\mcC_2$ of matroids are minor-closed,
then so is $\mcC_1\cup\mcC_2$.  If $M$ is an excluded minor for
$\mcC_1\cup\mcC_2$, then some minor of $M$ is an excluded minor for
$\mcC_1$ and another is an excluded minor for $\mcC_2$; furthermore,
no proper minor of $M$ has this property.  These remarks motivate the
following definition.  A matroid $M$ is an \emph{intertwine} of
matroids $M_1$ and $M_2$ if $M$ but none of its proper minors has both
an $M_1$-minor (i.e., a minor isomorphic to $M_1$) and an $M_2$-minor.
Thus, each excluded minor for $\mcC_1\cup\mcC_2$ is an intertwine of
some excluded minor for $\mcC_1$ and some excluded minor for $\mcC_2$.

Many important results and problems in matroid theory involve the
question of whether the set of excluded minors for a given
minor-closed class of matroids is finite; this leads to the question
of whether some pairs of matroids have infinitely many intertwines.
This question was raised by Tom Brylawski~\cite{const}; see
also~\cite[Problem 14.4.6]{ox}, where it is also attributed to Neil
Robertson and, in a different form, to Dominic Welsh.  The question
was settled affirmatively by Dirk Vertigan in the mid 1990's in
unpublished work; we sketch his construction in
Section~\ref{sec:dirk}.  Jim Geelen gave another
construction~\cite[Section 5]{jim}: for each pair of spikes, neither
being a minor of the other and all elements of which are in dependent
transversals, he constructed infinitely many intertwines that are also
spikes.  (That the class of spikes contains such infinite sets of
intertwines follows from Vertigan's construction along with his
embedding of the minor ordering on all matroids into that on the class
of spikes (for this intriguing embedding, see \cite[Section 3]{jim});
Geelen's construction is an attractive realization of this
phenomenon.)  In this paper, we take weaker hypotheses than the
earlier constructions used; we assume only that neither $M_1$ nor
$M_2$ can be obtained, up to isomorphism, from the other via free
extensions, free coextensions, and minors; for such a pair $(M_1,
M_2)$, we show that particular amalgams of certain free coextensions
of $M_1$ and $M_2$ are intertwines.  This yields many intertwines of
each sufficiently large rank; indeed, for some pairs, a variation on
our basic construction produces intertwines whose number grows at
least exponentially as a function of the rank.

We assume readers know basic matroid theory, an excellent account of
which is in~\cite{ox}.  Key background topics are collected in
Section~\ref{sec:pre} and the construction and a variation are given
in Section~\ref{sec:inter}.  In Section~\ref{sec:props}, we treat
properties of these intertwines; for instance, we show that for large
ranks, the intertwines we construct have large connectivity and
uniform minors of large rank and corank; we show that if both matroids
have no free elements, no cofree elements, no isthmuses, and no loops,
then, for a fixed integer $k$, the intertwines we construct cover the
full range of possible sizes for the ground sets of rank-$k$
intertwines of the pair; we also show that the construction preserves
certain properties, such as being transversal and being a gammoid.  In
Section~\ref{sec:dirk}, we explain the relation between our
construction and Dirk Vertigan's.

\section{Background}\label{sec:pre}

The intertwines we construct are defined via cyclic flats and their
ranks.  A \emph{cyclic set} in a matroid $M$ is a (possibly empty)
union of circuits.  The cyclic flats of $M$, ordered by inclusion,
form a lattice; indeed, $F\join G = \cl_M(F\cup G)$ and $F\meet G$ is
the union of the circuits in $F\cap G$.  We let $\mcZ(M)$ denote both
the set and the lattice of cyclic flats of $M$.  The following
well-known results are easy to prove~\cite[Problem 2.1.13]{ox}:
\begin{enumerate}
\item $\mcZ(M^*) = \{S-F\,:\,F\in \mcZ(M)\}$, where $S$ is the ground
  set of $M$,
\item $\mcZ(M_1\oplus M_2) = \{F_1\cup F_2\,:\, F_1\in \mcZ(M_1)
  \text{ and } F_2\in \mcZ(M_2)\}$, and
\item a matroid is determined by its cyclic flats and their ranks.
\end{enumerate}
There are many ways to prove property (3); for instance, one can show
how to get the circuits or the independent sets, or show that the rank
of an arbitrary set $Y$ in $M$ is given by the formula
\begin{equation}\label{rankbycyclicflats}
  r(Y)=\min\{r(F)+|Y-F|\,:\,F\in\mcZ(M)\}.
\end{equation}
The following result from~\cite{juliethesis,cyclic}
carries property (3) further.

\begin{prop}\label{prop:axioms}
  Let $\mcZ$ be a collection of subsets of a set $S$ and let $r$ be an
  integer-valued function on $\mcZ$.  There is a matroid for which
  $\mcZ$ is the collection of cyclic flats and $r$ is the rank
  function restricted to the sets in $\mcZ$ if and only if
  \begin{itemize}
  \item[(Z0)] $\mcZ$ is a lattice under inclusion,
  \item[(Z1)] $r(0_{\mcZ})=0$, where $0_{\mcZ}$ is the least element
    of $\mcZ$,
  \item[(Z2)] $0<r(Y)-r(X)<|Y-X|$ for all sets $X,Y$ in $\mcZ$ with
    $X\subset Y$, and
  \item[(Z3)] for all pairs of incomparable sets $X,Y$ in $\mcZ$,
    $$   r(X)+r(Y)\geq r(X\join Y) + r(X\meet Y) + |(X\cap Y) - (X\meet
    Y)|.$$
  \end{itemize}
\end{prop}

Recall that the free extension $M+x$ of the matroid $M$ on $S$ by the
element $x\not\in S$ is the matroid on $S\cup x$ whose circuits are
those of $M$ along with the sets $B\cup x$ as $B$ runs over the bases
of $M$.  We extend this notation to sets: $M+X$ is the result of
applying free extension iteratively to add all elements of $X$ to $M$.
From the perspective of Proposition~\ref{prop:axioms}, $M+X$, for
$X\ne \emptyset$, is the matroid on $S\cup X$ whose cyclic flats and
ranks are (i) the proper cyclic flats $F$ of $M$, with rank $r_M(F)$,
and (ii) $S\cup X$, of rank $r(M)$.  Dually, the cyclic flats and
ranks of the free coextension $M\times X = (M^*+X)^*$ are (i) the sets
$F\cup X$, of rank $r_M(F)+|X|$, for $F\in \mcZ(M)$ with $F\ne
\emptyset$, and (ii) the empty set, of rank $0$.

We use only the simplest type of lift and truncation.  The
\emph{$i$-fold lift} $L^i(M)$ of $M$ is $(M\times X)\del X$ where
$|X|=i$; dually, $(M+X)/X$ is the \emph{$i$-fold truncation,}
$T^i(M)$. 

The \emph{nullity} of a set $Y$ is $\eta(Y)=|Y|-r(Y)$.  Let $\mcZ'(M)$
be the set of nonempty proper cyclic flats of $M$ and let
$\eta(\mcZ'(M))$ be the sum of the nullities of these flats.  The
following lemma is easy to prove.

\begin{lemma}\label{lem:cy}
  If $F\in \mcZ(M\del x)$, then $\cl_M(F)\in \mcZ(M)$.
  If $\cl_M(F) = F$, then $\eta_{M\del x}(F)$ is $\eta_{M}(F)$; if
  $\cl_M(F) = F\cup x$, then $\eta_{M\del x}(F) = \eta_{M}(F\cup
  x)-1$.

  Dually, if $F\in \mcZ(M/y)$, then exactly one of $F$ and
  $F\cup y$ is in $\mcZ(M)$.  The nullities of $F$ in $M/y$ and
  the corresponding cyclic flat of $M$ agree unless $y$ is a loop of
  $M$, in which case $\eta_{M/y}(F) = \eta_{M}(F\cup y)-1$.

  Thus, if $N$ is a minor of $M$, then $\eta(\mcZ'(N)) \leq
  \eta(\mcZ'(M))$.
\end{lemma}

While a cyclic flat of a matroid may give rise to cyclic flats in its
restrictions, the next lemma identifies a situation in which this does
not happen.

\begin{lemma}\label{lem:del}
  Let $Z$ be a cyclic flat of $M$.  If a subset $U$ of $Z$ with
  $|U|\geq \eta(Z)$ is contained in all nonempty cyclic flats that are
  contained in $Z$, then $Z-U$ is independent.
\end{lemma}

\begin{proof}
  Assume, to the contrary, that $Z-U$ contains some circuit $C$.  The
  nonempty cyclic flat $\cl(C)$ is contained in $Z$, so $U\subseteq
  \cl(C)$.  Thus $U\subseteq \cl(Z-U)=Z$.  Now $\eta(Z-U)>0$ and
  $U\subseteq \cl(Z-U)$ give $\eta(Z)>|U|$, contrary to the assumed
  inequality.
\end{proof}

For matroids $M_1$ on $S_1$ and $M_2$ on $S_2$, a matroid $M$ on
$S_1\cup S_2$ with $M|S_1=M_1$ and $M|S_2=M_2$ is called an
\emph{amalgam} of $M_1$ and $M_2$.

An element $x$ in a matroid $M$ is \emph{free in} $M$ if $M = (M\del
x)+x$.  Dually, an element $y$ is \emph{cofree in} $M$ if $M =
(M/y)\times y$.  Let $FI(M)$ be the set of all elements of $M$ that
are in no proper cyclic flat of $M$; thus, $FI(M)$ consists of the
free elements and isthmuses of $M$.  Note that $FI(M^*)$ is the
intersection of all nonempty cyclic flats of $M$; it consists of the
cofree elements and loops of $M$.

\section{Intertwines}\label{sec:inter}

We now construct the matroids of interest.  The notation established
in this paragraph is used in the rest of the paper.  Assume the
matroids $M_1$ and $M_2$ have positive rank and are defined on
disjoint ground sets, $S_1$ and $S_2$, respectively.  Let $r_1$ and
$r_2$ be their rank functions, and let $\eta_1$ and $\eta_2$ be their
nullity functions.  Fix subsets $S'_1$ of $S_1$ and $S'_2$ of $S_2$,
an integer $k$ with
\begin{equation}\label{ineq:k}
  k\geq r(M_1)+\eta_1(S'_1) + r(M_2) + \eta_2(S'_2),
\end{equation}
and sets $T_1$ and $T_2$ with
\begin{equation}\label{eq:|T_i|}
  |T_1|=k-r(M_1)-|S'_2| \qquad \text{and} \qquad |T_2|=k-r(M_2)-|S'_1|
\end{equation}
where $T_1$, $T_2$, and $S_1\cup S_2$ are mutually disjoint.  Let
$$\mcZ = \mcZ'\bigl(M_1\times (T_1\cup S'_2)\bigr) \cup
\mcZ'\bigl(M_2\times (T_2\cup S'_1)\bigr)\cup \{\emptyset,S_1\cup
S_2\cup T_1\cup T_2\}.$$ (Note that inequality~(\ref{ineq:k}) gives
$|T_1|+|S'_2|\geq \eta_1(S'_1)+r(M_2)+\eta_2(S'_2)$, which is
positive; therefore $M_1\times (T_1\cup S'_2)$ is a proper coextension
of $M_1$ and so has no loops.  Likewise $M_2\times (T_2\cup S'_1)$ has
no loops.  Thus, the least cyclic flat of these matroids is
$\emptyset$.)  Define $r:\mcZ\rightarrow \mathbb{Z}$ by
\begin{enumerate}
\item $r(F\cup T_1\cup S'_2) = r_1(F)+|T_1|+|S'_2|$ for $F\in
  \mcZ'(M_1)$,
\item $r(F\cup T_2\cup S'_1) = r_2(F)+|T_2|+|S'_1|$ for $F\in
  \mcZ'(M_2)$,
\item $r(\emptyset) = 0$, and
\item $r(S_1\cup S_2\cup T_1\cup T_2)=k$.
\end{enumerate}

\begin{thm}\label{thm:propshold}
  The pair $(\mcZ,r)$ satisfies properties (Z0)-(Z3) of
  Proposition~\ref{prop:axioms}.  The rank-$k$ matroid $M$ on $S_1\cup
  S_2\cup T_1\cup T_2$ thus defined has the following properties:
  \begin{itemize}
  \item[(i)] $M$ is an amalgam of $M_1\times (T_1\cup S'_2)$ and
    $M_2\times (T_2\cup S'_1)$, and
  \item[(ii)] $\eta_1(F) = \eta_M(F\cup T_1\cup S'_2)$ for $F\in
    \mcZ'(M_1)$ and $\eta_2(F) = \eta_M(F\cup T_2\cup S'_1)$ for $F\in
    \mcZ'(M_2)$.
  \end{itemize}
\end{thm}

\begin{proof}
  Property (Z0) holds since any pair of sets in $\mcZ$ that does not
  have a join in one of $\mcZ'\bigl(M_1\times (T_1\cup S'_2)\bigr)$
  and $\mcZ'\bigl(M_2\times (T_2\cup S'_1)\bigr)$ has $S_1\cup S_2\cup
  T_1\cup T_2$ as the join.  Property (Z1) is item (3) above.
  Property (Z2) follows from this property in $\mcZ\bigl(M_1\times
  (T_1\cup S'_2)\bigr)$ and $\mcZ\bigl(M_2\times (T_2\cup
  S'_1)\bigr)$, as do all instances of property (Z3) except in the
  case $X = F_1\cup T_1\cup S'_2$ with $F_1\in \mcZ(M_1)$ and $Y =
  F_2\cup T_2\cup S'_1$ with $F_2\in \mcZ(M_2)$.  In this case, the
  required inequality is
  $$r_1(F_1)+|T_1|+|S'_2| + r_2(F_2)+|T_2|+|S'_1|\geq k +
  |F_1\cap S'_1| + |F_2\cap S'_2|,$$ which follows from
  inequality~(\ref{ineq:k}) and equations~(\ref{eq:|T_i|}).

  By symmetry, assertion (i) follows if we show that for any $F\in
  \mcZ'(M_2)$, the difference $(F\cup T_2\cup
  S'_1)-(T_2\cup(S_2-S'_2))$ is independent in $M|S_1\cup T_1\cup
  S'_2$.  All such differences are contained in $S'_1\cup S'_2$, so it
  suffices to show that $S'_1\cup S'_2$ is independent in $M$.  To
  show this, by equation~(\ref{rankbycyclicflats}) it suffices to
  prove $$r_M(Z) + |(S'_1\cup S'_2)-Z|\geq |S'_1|+|S'_2|$$ for all
  $Z\in \mcZ(M)$; again by symmetry, it suffices to show this for
  $Z=F_1\cup T_1\cup S'_2$ with $F_1\in \mcZ'(M_1)$.  For such $Z$,
  the required inequality is
  $$r_1(F_1) + |T_1|+|S'_2| + |S'_1-F_1|\geq |S'_1|+|S'_2|,$$ or,
  using equations~(\ref{eq:|T_i|}) and manipulating,
  $$k \geq r(M_1) + |S'_1\cap F_1| - r_1(F_1)
  + |S'_2|.$$ This inequality follows from inequality~(\ref{ineq:k})
  since $|S'_2|\leq r(M_2)+\eta_2(S'_2)$ and
  \begin{align*}
    |S'_1\cap F_1| - r_1(F_1) \leq & \, |S'_1\cap F_1| -
    r_1(S'_1\cap F_1) \\
    = & \, \eta_1(S'_1\cap F_1) \\
    \leq & \, \eta_1(S'_1).
  \end{align*}
   
 Assertion (ii) is evident.
\end{proof}

The matroid so constructed depends on $M_1$, $M_2$, $k$, $S'_1$,
$S'_2$, $T_1$, and $T_2$.  If (as in the next result) listing all
parameters aids clarity, we use $M_k(M_1,S'_1,T_1;M_2,S'_2,T_2)$ to
denote this matroid; otherwise we simply write $M$.

The next result, which follows by comparing the cyclic flats and their
ranks, shows that combining the construction with duality yields other
instances of the same construction.

\begin{thm}\label{thm:dual}
  With $j=|S_1|+|S_2|+|T_1|+|T_2|-k$, we have
  $$\bigl(M_k(M_1,S'_1,T_1;M_2,S'_2,T_2)\bigr)^* =
  M_j(M^*_1,S_1-S'_1,T_2;M^*_2,S_2-S'_2,T_1).$$ Also, $j\geq
  r(M^*_1)+\eta_{M^*_1}(S_1-S'_1) + r(M^*_2) + \eta_{M^*_2}(S_2-S'_2)$
  if and only if $k$ satisfies inequality~(\ref{ineq:k}).
\end{thm}

We now treat the main result.  A similar but somewhat longer argument
would modestly increase the range for $k$; we opt for the shorter
proof since the main interest is in having infinitely many
intertwines.  Recall that $FI(M)$ is the set of free elements and
isthmuses of $M$, so $FI(M^*)$ is the set of cofree elements and loops
of $M$.

\begin{thm}\label{thm:inter}
  Assume that the ground sets $S_1$ and $S_2$ of $M_1$ and $M_2$ are
  disjoint and that no matroid isomorphic to $M_1$ (resp., $M_2$) can
  be obtained from $M_2$ (resp., $M_1$) by any combination of minors,
  free extensions, and free coextensions.  For $i\in \{1,2\}$, fix a
  set $S'_i$ with $FI(M_i)\subseteq S'_i\subseteq S_i-FI(M_i^*)$.  If
  $k\geq 4\max\{|S_1|,|S_2|\}$, then the matroid $M$ defined above is
  an intertwine of $M_1$ and $M_2$.
\end{thm}

\begin{proof}
  Theorem~\ref{thm:propshold} part (i) shows that $M_1$ and $M_2$ are
  minors of $M$.  By symmetry, to prove that $M$ is an intertwine, it
  suffices to show that for $a\in S_1\cup T_1$, neither $M\del a$ nor
  $M/a$ has both an $M_1$-minor and an $M_2$-minor; furthermore, by
  Theorem~\ref{thm:dual} and the observation that the hypotheses are
  invariant under duality, it suffices to treat $M\del a$.  Now
  $|T_1|\geq r(M_1^*)+r(M_2)+|S_i|+1$ since $k\geq
  4\max\{|S_1|,|S_2|\}$.  If $M\del a\del X /Y\simeq M_i$ with
  $i\in\{1,2\}$, then $M\del a\del X /Y$ has $|S_i|$ elements, at
  least $|T_1| - |X\cap T_1| - |Y\cap T_1| -1$ of which are in $T_1$,
  so $|X\cap T_1| + |Y\cap T_1| \geq |T_1|-|S_i|-1$, and therefore
  $$|X\cap T_1| + |Y\cap T_1| \geq r(M^*_1)+r(M_2).$$ Thus, either (i)
  $|X\cap T_1|\geq r(M^*_1)$ or (i$^*$) $|Y\cap T_1|\geq r(M_2)$.

  We claim that the three conclusions below follow when inequality (i)
  holds:
  \begin{enumerate}
  \item[(1)] $M\del (X\cap T_1)=\bigl(M_2\times(T_2\cup
    S'_1)\bigr)+\bigl((S_1-S'_1)\cup(T_1-X)\bigr)$,
  \item[(2)] $M\del (X\cap T_1)$ has no $M_1$-minor, and
  \item[(3)] $\eta_2(F) = \eta_{M\del (X\cap T_1)}(F\cup T_2\cup
    S'_1)$ for $F\in \mcZ'(M_2)$.
  \end{enumerate}
  Item (1) holds since, using Lemma~\ref{lem:del}, we get that the
  proper cyclic flats and their ranks in the two matroids agree.  Item
  (1) and the hypotheses give item (2).  Item (3) is immediate.

  Inequalities (i) and (i$^*$) are related by duality, so
  Theorem~\ref{thm:dual} and the results in the last paragraph give
  the following conclusion if inequality (i$^*$) holds:
  \begin{enumerate}
  \item[(1$^*$)] $M/(Y\cap T_1)=\bigl(M_1\times\bigl((T_1-Y)\cup
    S'_2\bigr)\bigr)+\bigl((S_2-S'_2)\cup T_2\bigr)$,
  \item[(2$^*$)] $M/(Y\cap T_1)$ has no $M_2$-minor, and
  \item[(3$^*$)] $\eta_1(F) = \eta_{M/(Y\cap T_1)}\bigl(F\cup
    (T_1-Y)\cup S'_2\bigr)$ for $F\in \mcZ'(M_1)$.
  \end{enumerate}
 
  For $a\in (S_1-S'_1)\cup T_1$, assume $M\del a$ has an $M_1$-minor,
  say $M\del a \del X/Y$.  By item (2), inequality (i$^*$) holds.
  Since $a$ is in at least one set in $\mcZ'(M/(Y\cap T_1))$, item
  (3$^*$) gives $\eta\bigl(\mcZ'(M/(Y\cap T_1)\del a)\bigr)<
  \eta(\mcZ'(M_1))$; the contradiction $M\del a\del X/Y \not \simeq
  M_1$ now follows from Lemma~\ref{lem:cy}.

  Lastly, for $a\in S'_1$, assume $M\del a$ has an $M_2$-minor, say
  $M\del a \del X/Y$.  Inequality (i) holds by item (2$^*$).  Now
  $a\in S'_1$ gives $\eta\bigl(\mcZ'(M\del (X\cap T_1)\del a)\bigr)<
  \eta(\mcZ'(M_2))$, which, with Lemma~\ref{lem:cy}, gives the
  contradiction $M\del a\del X/Y \not \simeq M_2$.
\end{proof}

Assume $FI(M_i)=\emptyset=FI(M_i^*)$ for $i\in \{1,2\}$.  Reflecting
on the proof above shows that $a\in S_1$ if and only if neither $M\del
a$ nor $M/a$ has an $M_1$-minor, and likewise for $S_2$ and $M_2$.
These conclusions and the structure of the cyclic flats of $M$ show
that the counterparts of the sets $S_1$, $S_2$, $T_1$, $T_2$, $S'_1$,
and $S'_2$ can be determined from any matroid that is isomorphic to
$M$.  This gives the following result.

\begin{cor}
  Assume $FI(M_i)=\emptyset=FI(M_i^*)$ for $i\in \{1,2\}$.  The
  construction gives at least $(|S_1|+1)(|S_2|+1)$ nonisomorphic
  rank-$k$ intertwines of $M_1$ and $M_2$ for each integer $k\geq
  4\max\{|S_1|,|S_2|\}$.  If, in addition, both $M_1$ and $M_2$ have
  trivial automorphism groups, then the construction yields
  $2^{|S_1|+|S_2|}$ nonisomorphic rank-$k$ intertwines.
\end{cor}

Knowing more about $M_1$ and $M_2$ may suggest variations on the
construction that yield more intertwines, as we now illustrate.
Assume that in addition to satisfying the conditions in
Theorem~\ref{thm:inter}, neither $M_1$ nor $M_2$ has
circuit-hyperplanes.  Let $M$ be the intertwine constructed above.
From the bound on $k$ in Theorem~\ref{thm:inter} we get $|T_1\cup
T_2|>k$.  Let $\mcH$ be a collection of $k$-subsets of $T_1\cup T_2$
with $|H\cap H'|\leq k-2$ whenever $H$ and $H'$ are distinct sets in
$\mcH$.  In the construction, replace $\mcZ$ by $\mcZ\cup \mcH$ and
extend $r$ to $\mcZ\cup \mcH$ by setting $r(H)=k-1$ for all $H\in
\mcH$.  Properties (Z0)--(Z3) of Proposition~\ref{prop:axioms} are
easily verified.  Let $M'$ be the matroid thus constructed.  The sets
in $\mcH$ are the circuit-hyperplanes of $M'$.  By comparing the
cyclic flats and their ranks, it follows that if $M'\del X/Y$ has no
circuit-hyperplanes, then $M'\del X/Y = M\del X/Y$.  Since neither
$M_1$ nor $M_2$ has circuit-hyperplanes, it follows that if some
single-element deletion or contraction of $M'$ had both an $M_1$-minor
and an $M_2$-minor, then the same would be true of the corresponding
single-element deletion or contraction of $M$, contrary to
Theorem~\ref{thm:inter}.  Thus, $M'$ is an intertwine of $M_1$ and
$M_2$.  Thus, we have the following result.

\begin{thm}\label{thm:chvar}
  Assume $M_1$ and $M_2$ satisfy the hypotheses of
  Theorem~\ref{thm:inter} and neither has circuit-hyperplanes.  For
  each integer $n$, there is an integer $k_0$ so that if $k\geq k_0$,
  then $M_1$ and $M_2$ have at least $n$ intertwines of rank $k$.
\end{thm}

To take these ideas a step further, we give a simple proof that, as
$k$ grows, the number of nonisomorphic intertwines arising from the
variation on the construction grows at least exponentially.  To
simplify the discussion slightly, assume both $|T_1\cup T_2|$ and $k$
are even.  Let $\mbH$ be the set of all sets $\mcH$ of $k$-subsets of
$T_1\cup T_2$ such that $|H\cap H'|\leq k-2$ whenever $H$ and $H'$ are
distinct sets in $\mcH$.  One way to get a set $\mcH$ in $\mbH$ is to
pair off the elements in $T_1\cup T_2$ and, to get each set in $\mcH$,
choose $k/2$ pairs.  Even among sets $\mcH$ formed in this limited
way, their maximal size grows exponentially as a function of $k$ (much
as $\binom{2n}{n}$ grows exponentially as a function of $n$).  Among
all sets in $\mbH$, let $\mcH$ be one of maximal size.  Subsets of
$\mcH$ of different sizes give rise to nonisomorphic intertwines
(their numbers of circuit-hyperplanes differ), so these intertwines
demonstrate our claim.

This discussion and the last two results suggest several problems.
Let $i(k;M_1,M_2)$ denote the number of rank-$k$ intertwines of $M_1$
and $M_2$ up to isomorphism.  What can be said about $i(k;M_1,M_2)$?
If $M_1$ and $M_2$ satisfy the hypotheses of Theorem~\ref{thm:inter},
is $i(k;M_1,M_2)$ increasing as a function of $k$, at least for
sufficiently large $k$?  If so, under what conditions on $M_1$ and
$M_2$ is the difference $i(k+1;M_1,M_2) - i(k;M_1,M_2)$ bounded above
by a constant or by a polynomial?  Under what conditions does
$i(k;M_1,M_2)$ grow exponentially or super-exponentially?

A matroid $M$ is a \emph{labelled intertwine} of $M_1$ and $M_2$ if
$M$ but none of its proper minors has minors equal to $M_1$ and $M_2$.
We end this section by showing that weaker hypotheses than those in
Theorem~\ref{thm:inter} suffice for our construction to yield labelled
intertwines.

\begin{thm}\label{thm:label}
  Assume $S_1$ and $S_2$ are disjoint.  If inequality~(\ref{ineq:k})
  holds, neither $M_1$ nor $M_2$ is uniform, $\mcZ'(M_1) \ne\{S'_1\}$,
  and $\mcZ'(M_2) \ne\{S'_2\}$, then the matroid $M$ constructed above
  is a labelled intertwine of $M_1$ and $M_2$.
\end{thm}

\begin{proof}
  By symmetry, to prove that no proper minor of $M$ has both $M_1$ and
  $M_2$ as minors, it suffices to show that if $M\del X/Y=M_1$, then
  $X=(S_2-S'_2)\cup T_2$ and $Y=T_1\cup S'_2$.  Thus, assume $M\del
  X/Y=M_1$.  Fix $F\in \mcZ'(M_1) -\{S'_1\}$.  By Lemma~\ref{lem:cy},
  the cyclic flat $F$ of $M\del X/Y$ must arise from the cyclic flat
  $F\cup T_1\cup S'_2$ of $M$; from Theorem~\ref{thm:propshold} part
  (ii), it follows that for $M\del X/Y$ to yield the same nullity on
  $F$ as in $M_1$, each element of $T_1\cup S'_2$ must be contracted;
  dually, each element of $(S_2-S'_2)\cup T_2$ must be deleted.  Thus,
  $X=(S_2-S'_2)\cup T_2$ and $Y=T_1\cup S'_2$.
\end{proof}

\section{Further Results}\label{sec:props}

Among the pairs of matroids that Theorem~\ref{thm:inter} applies to
are any two spikes of rank at least $4$, neither of which is a minor
of the other, provided that the one of smaller rank (if the ranks
differ) is not a free spike.  (We use the definition of spikes
in~\cite{jim}, which some sources call tip-less spikes. Free spikes
are the only spikes that can be obtained from a spike by minors along
with at least one lift or truncation.)  Thus, the assumption in the
construction in~\cite{jim} that each element is in a dependent
transversal is not needed here.  However, unlike the construction
in~\cite{jim}, the intertwine we get when $M_1$ and $M_2$ are spikes
is not a spike.

The construction here and that in~\cite{jim} give intertwines with
contrasting properties and so show that some properties that hold for
one construction need not hold for intertwines in general.  For
instance, the intertwines constructed here have neither small circuits
nor small cocircuits, but those constructed in~\cite{jim} have each
element in a many $4$-circuits and in many $4$-cocircuits.  Also, in
our construction the number of cyclic flats does not depend on the
rank, but in the construction in~\cite{jim} the number of cyclic flats
grows with the rank (as is true for the variation we discussed before
Theorem~\ref{thm:chvar}).

\subsection{Sizes of intertwines}
We show that the intertwines constructed above can exhibit the full
range of possible sizes for each rank.

\begin{thm}
  If $S$ is the ground set of a rank-$k$ intertwine of $M_1$ and
  $M_2$, then $$2k-r(M_1)-r(M_2)\leq |S| \leq 2k+r(M^*_1)+r(M^*_2).$$
  If $FI(M_i)=\emptyset=FI(M_i^*)$ for $i\in \{1,2\}$, then the
  construction in Section~\ref{sec:inter} gives intertwines of each
  cardinality in this range.
\end{thm}

\begin{proof}
  Let $M$, on the set $S$, be a rank-$k$ intertwine of $M_1$ and
  $M_2$, so $M\del X/Y\simeq M_1$ and $M\del X'/Y'\simeq M_2$ for some
  subsets $X$, $Y$, $X'$, $Y'$ of $S$.  Standard arguments about
  minors show that we may assume that $Y$ and $Y'$ are independent
  sets with $|Y| = k-r(M_1)$ and $|Y'| = k-r(M_2)$.  No proper
  contraction of $M$ has both $M_1$- and $M_2$-minors, so $Y\cap Y' =
  \emptyset$; thus, $|Y|+|Y'| \leq |S|$, so $k-r(M_1) + k-r(M_2) \leq
  |S|$.  This lower bound on $|S|$ is achieved in the construction
  when $S'_1 = S_1$ and $S'_2 = S_2$.  No proper deletion of $M$ has
  both $M_1$- and $M_2$-minors, so $X\cap X' = \emptyset$; thus,
  $|X|\leq |S_2|+|Y'|$.  This inequality, the equation $|S| =
  |S_1|+|X| + |Y|$, and values for $|Y|$ and $|Y'|$ give the upper
  bound.  This bound in attained when $S'_1 = \emptyset = S'_2$.  By
  varying $|S'_1|$ and $|S'_2|$, all cardinalities between these
  bounds can be realized.
\end{proof}

\subsection{Representable matroids}
All spikes are contained in $\mcE(U_{2,6},U_{4,6})$, the class of
matroids that have neither $U_{2,6}$- nor $U_{4,6}$-minors.  The
results in this subsection and the next are akin to a corollary that
Vertigan got from his work on intertwines and spikes: some pairs of
matroids in $\mcE(U_{2,6},U_{4,6})$ have infinitely many intertwines
in $\mcE(U_{2,6},U_{4,6})$.

The result below uses the following equivalent formulations of two
special cases of our construction.  (The first assertion follows by
comparing the cyclic flats and their ranks; the second is the dual of
the first.  Recall that $T^k$ and $L^j$ denote truncations and lifts.)
If $k\geq r(M_1)+r(M_2)$, then
$$M_k(M_1,\emptyset,T_1;M_2,\emptyset,T_2)=T^k\bigl((M_1\times T_1)
\oplus (M_2\times T_2)\bigr).$$ If $k\geq r(M_1)+|S_1|+r(M_2)+|S_2|$
and $j=k-r(M_1)-r(M_2)$, then
$$M_k(M_1,S_1,T_1;M_2,S_2,T_2)=L^j\bigl((M_1+T_2) \oplus
(M_2+T_1)\bigr).$$

\begin{cor}
  Assume a class $\mcC$ of matroids is closed under direct sum, free
  extension, free coextension, truncation, and lift.  If $M_1,M_2\in
  \mcC$ satisfy the hypotheses of Theorem~\ref{thm:inter} and if
  either $FI(M_1)=\emptyset=FI(M_2)$ or
  $FI(M^*_1)=\emptyset=FI(M^*_2)$, then $\mcC$ contains infinitely
  many intertwines of $M_1$ and $M_2$.
\end{cor}

Such classes $\mcC$ include the class of matroids that are
representable over a given infinite field and the class of matroids
that are representable over a given characteristic.

\subsection{Transversal  matroids and gammoids}
We next show that the intertwine $M$ that we constructed is
transversal if and only if $M_1$ and $M_2$ are; we also treat the
corresponding statements for several related types of matroids.  We
will use the characterization of transversal matroids in
Lemma~\ref{lem:masoningleton}, which is due to
Ingleton~\cite{ingleton} and refines a result of Mason.  For a
collection $\mcF$ of sets, let $\cap\mcF$ be $\bigcap_{X\in \mcF}X$
and $\cup\mcF$ be $\bigcup_{X\in \mcF}X$.

\begin{lemma}\label{lem:masoningleton}
  A matroid $M$ is transversal if and only if for all $\mcA\subseteq
  \mcZ(M)$ with $\mcA\ne\emptyset$,
  \begin{equation}\label{mi}
    \sum_{\mcF\subseteq \mcA} (-1)^{|\mcF|+1} r(\cup\mcF)\geq r(\cap\mcA). 
  \end{equation}
\end{lemma}

In this result, it suffices to consider only antichains $\mcA$ of
cyclic flats since if $X,Y\in \mcA$ with $X\subset Y$, then using
$\mcA - \{Y\}$ in place of $\mcA$ does not change either side of
inequality (\ref{mi}); with $\mcA$, the terms on the left side that
include $Y$ cancel via the involution that adjoins $X$ to, or omits
$X$ from, $\mcF$.  Also, it suffices to focus on inequality (\ref{mi})
for $|\mcA|>2$ since equality holds when $|\mcA| = 1$ and the case of
$|\mcA| = 2$ is the semimodular inequality.

\begin{cor}\label{cor:trcoex}
  A matroid $M$ is transversal if and only if $M\times x$ is.
\end{cor}

Cotransversal matroids are duals of transversal matroids.
Bitransversal matroids are both transversal and cotransversal.
Gammoids are minors of transversal matroids.  Restrictions of
transversal matroids are transversal, so any gammoid is a contraction
of some transversal matroid; it follows that any gammoid is a
nullity-preserving contraction of some transversal matroid.  The class
of gammoids is closed under duality, so any gammoid has a
rank-preserving extension to a cotransversal matroid.

\begin{thm}\label{thm:transv}
  Assume inequality~(\ref{ineq:k}) holds.  The matroids $M_1$ and
  $M_2$ are transversal if and only if
  $M=M_k(M_1,S'_1,T_1;M_2,S'_2,T_2)$ is.  The corresponding statements
  hold for cotransversal matroids, bitransversal matroids, and
  gammoids.
\end{thm}

\begin{proof}
  Since $M_1\times (T_1\cup S'_2)$ and $M_2\times (T_2\cup S'_1)$ are
  restrictions of $M$, from Corollary~\ref{cor:trcoex} it follows that
  if $M$ is transversal, then so are $M_1$ and $M_2$.  Now assume
  $M_1$ and $M_2$ are transversal.  Let $\mcA$ be an antichain in
  $\mcZ(M)$ with $|\mcA|>2$.  Set $$\mcA_1 = \mcA \cap \mcZ'(M_1\times
  (T_1\cup S'_2)) \quad \text{ and } \quad \mcA_2 = \mcA \cap
  \mcZ'(M_2\times (T_2\cup S'_1)).$$ Thus, $\mcA$ is the disjoint
  union of $\mcA_1$ and $\mcA_2$.  By Corollary~\ref{cor:trcoex} and
  Lemma~\ref{lem:masoningleton}, inequality (\ref{mi}) holds for
  $\mcA_1$ if it is nonempty, and likewise for $\mcA_2$; thus, this
  inequality holds for $\mcA$ if one of $\mcA_1$ and $\mcA_2$ is
  empty.  Assume neither is empty.  For $F_1\in \mcA_1$ and $F_2\in
  \mcA_2$, we have $r_M(F_1\cup F_2) = r(M)$, so
  \begin{align*}
    \sum_{\mcF\subseteq\mcA} (-1)^{|\mcF|+1} r(\cup\mcF) = \, &
    \sum_{\mcF_1\subseteq\mcA_1} (-1)^{|\mcF_1|+1} r(\cup\mcF_1) +
    \sum_{\mcF_2\subseteq\mcA_2} (-1)^{|\mcF_2|+1}
    r(\cup\mcF_2)    \\
    & \qquad + \sum\limits_{\substack{\mcF_1\subseteq\mcA_1,\,
        \mcF_1\ne\emptyset \\ \mcF_2\subseteq\mcA_2,\,
        \mcF_2\ne\emptyset}} (-1)^{|\mcF_1|+|\mcF_2|+1}\, r(M)\\
    \geq \, & r(\cap\mcA_1) + r(\cap\mcA_2) - r(M) \\
    \geq \,& r(\cap\mcA),
  \end{align*}
  where the last line follows from semimodularity along with the
  inclusions $T_1\cup S'_2\subseteq \cap\mcA_1$ and $T_2\cup
  S'_1\subseteq \cap\mcA_2$, and the fact that $T_1\cup T_2\cup
  S'_1\cup S'_2$ spans $M$ (a consequence of
  equation~(\ref{rankbycyclicflats}) and inequality~(\ref{ineq:k})).
  Thus, inequality (\ref{mi}) holds, so $M$ is transversal. 

  The assertions about cotransversal and bitransversal matroids follow
  by Theorem~\ref{thm:dual}.

  If $M$ is a gammoid, then so are its minors $M_1$ and $M_2$.  Now
  assume $M_1$ and $M_2$ are gammoids.  Let $M'_1$ and $M'_2$ be
  rank-preserving cotransversal extensions of $M_1$ and $M_2$.  Thus,
  $M_k(M'_1,S'_1,T_1;M'_2,S'_2,T_2)$ is cotransversal since
  inequality~(\ref{ineq:k}) holds with $M'_1$ and $M'_2$ in place of
  $M_1$ and $M_2$.  Comparing the cyclic flats and their ranks shows
  that $M$ is a restriction of $M_k(M'_1,S'_1,T_1;M'_2,S'_2,T_2)$, so
  $M$ is a gammoid.
\end{proof}

\begin{cor}
  If $M_1$ and $M_2$ satisfy the hypotheses of Theorem~\ref{thm:inter}
  and are transversal, then infinitely many intertwines of $M_1$ and
  $M_2$ are transversal.  The corresponding statements hold for
  cotransversal matroids, bitransversal matroids, and gammoids.
\end{cor}

\subsection{Uniform minors}
We claim that $M|B_1\cup B_2\cup T_1\cup T_2$, where $B_i$ is a basis
of $M_i$, is the uniform matroid $U_{k,2k-|S'_1|-|S'_2|}$.  To see
this, note that if $C$ were a circuit in this restriction with
$r(C)<k$, then $\cl_M(C)\in\mcZ'(M)$; however, it follows from the
construction that flats in $\mcZ'(M)$ intersect $B_1\cup B_2\cup
T_1\cup T_2$ in independent sets.

\begin{cor}\label{cor:unif}
  If $M_1$ and $M_2$ satisfy the hypotheses of
  Theorem~\ref{thm:inter}, then for any integer $n$, some intertwine
  of $M_1$ and $M_2$ has a $U_{n,2n}$-minor.
\end{cor}

\subsection{Connectivity}
Recall that for any non-uniform matroid, $\lambda(M)\leq \kappa(M)$
where $\lambda(M)$ is the (Tutte) connectivity of $M$ and $\kappa(M)$
is its vertical connectivity.  Thus, showing that the connectivity of
intertwines can be arbitrarily large gives the counterpart for
vertical connectivity.

Qin~\cite{hq} proved $ \lambda((M+p)\times q) - \lambda(M)\in\{1,2\}$
for any matroid $M$.  For the matroid $M$ constructed above, fix a
subset $T'_2$ of $T_2$ with $|T'_2| = \eta(M_2)$.  Comparing the
cyclic flats and their ranks, with the help of Lemma~\ref{lem:del},
gives
$$M\del T'_2 = \bigl(M_1\times (T_1\cup
S'_2)\bigr)+\bigl((T_2-T'_2)\cup (S_2 - S'_2)\bigr).$$ After some
number of free coextensions or free extensions of $M_1$ according to
the difference between $|T_1\cup S'_2|$ and $|(T_2-T'_2)\cup (S_2 -
S'_2)|$ (which does not change as $k$ increases), the deletion $M\del
T'_2$ can be seen as resulting from free extension/free coextension
pairs, so the connectivity of such deletions $M\del T'_2$ grows with
$k$.  Since extending as needed by the elements in $T'_2$ to obtain
$M$ preserves the rank and introduces no circuits of size
$|T_2|+|S'_1|$ or smaller, $\lambda(M)$ also grows with $k$.

\begin{cor}
  If $M_1$ and $M_2$ satisfy the hypotheses of
  Theorem~\ref{thm:inter}, then for any integer $n$, some intertwine
  of $M_1$ and $M_2$ is $n$-connected.
\end{cor}

With the truncation that cuts the rank of the direct sum in half, it
follows that the intertwine $T^k\bigl((M_1\times T_1) \oplus
(M_2\times T_2)\bigr)$ (arising from $S'_1=\emptyset=S'_2$) is
rounded, that is, the ground set is not the union of two proper flats,
or, equivalently, each cocircuit spans. (This notion, also called
non-splitting, is equivalent to having $\kappa(M)=r(M)$.)

Note that in a rank-$n$ spike $M$ with $n\geq 4$, if $H$ is a
hyperplane spanned by $n-2$ legs (using the terminology
of~\cite{jim}), then $(H,E(M)-H)$ is a vertical $3$-separation of $M$.
Thus, the construction in this paper and that in~\cite{jim} yield
intertwines with contrasting connectivity and vertical connectivity
properties.

\section{The Relation to Vertigan's Construction}\label{sec:dirk}

As mentioned in the introduction, the first construction of infinite
sets of intertwines for pairs of matroids was given by Dirk Vertigan.
In this section we briefly outline his construction and show that,
although the approaches differ, some instances of the two
constructions coincide; furthermore, both approaches can be extended
to yield the same collections of intertwines.  Vertigan's theorem is
as follows.

\begin{thm}\label{thm:dirk}
  Assume neither $M_1$ nor $M_2$ can be obtained, up to isomorphism,
  from the other by any combination of minors, free extensions, and
  free coextensions.  If $FI(M_i)=\emptyset=FI(M_i^*)$ for $i\in
  \{1,2\}$, then $M_1$ and $M_2$ have infinitely many intertwines.
\end{thm}

The intertwines he constructed to prove this result are defined as
follows.  Let $S_1$ and $S_2$ be the ground sets of $M_1$ and $M_2$,
which, in contrast to Theorem~\ref{thm:inter}, need not be disjoint.
Let $X$ and $Y$ be disjoint $k$-element sets, where $k\geq
10\max\{|S_1|,|S_2|\}$, such that (i) $S_1\cup S_2\subseteq X\cup Y$,
(ii) $X\cap S_1$ has $r(M_1)$ elements and is dependent in $M_1$, and
(iii)~$Y\cap S_2$ has $r(M_2)$ elements and is dependent in $M_2$.
Set
$$M'_1 = \bigl(M_1+(Y-S_1)\bigr)\times (X-S_1) \quad \text{and} \quad
M'_2 = \bigl(M_2+(X-S_2)\bigr)\times (Y-S_2).$$ Thus,
$r(M'_1)=k=r(M'_2)$.  He argues that the intersection of the
collections of bases of $M'_1$ and $M'_2$ is the collection of bases
of a matroid on $X\cup Y$, and that this matroid is an intertwine of
$M_1$ and $M_2$.  Thus, this intertwine has rank $k$ and has $2k$
elements.  Vertigan observed that, as in Corollary~\ref{cor:unif},
these intertwines have uniform minors of large rank and corank.

To relate this construction to ours, we first show that the bases of
the intertwines we constructed can be described in a similar manner.
Using the notation in Section~\ref{sec:inter}, set
$$M_1''=\bigl(M_1\times (T_1\cup S'_2)\bigr)+\bigl(T_2\cup
(S_2-S'_2)\bigr)$$ and $$M_2''=\bigl(M_2\times (T_2\cup
S'_1)\bigr)+\bigl(T_1\cup (S_1-S'_1)\bigr).$$ Both $M''_1$ and $M''_2$
have rank $k$.  Observe that $\mcZ(M) = \mcZ(M''_1)\cup \mcZ(M''_2)$.
Using equation~(\ref{rankbycyclicflats}), it follows that a subset of
$S_1\cup S_2\cup T_1\cup T_2$ is a basis of $M$ if and only if it is a
basis of both $M''_1$ and $M''_2$.  In particular, the constructions
coincide when applied under the same set up, and the basis approach
can be extended to cover the results in this paper.  In the other
direction, it is easy to check that if we replace
inequality~(\ref{ineq:k}) with a slightly stronger inequality, then
Theorem~\ref{thm:propshold} applies even when $S_1$ and $S_2$ are not
disjoint; of course, then we need $S'_1\subseteq S_1 - S_2$ and
$S'_2\subseteq S_2 - S_1$.  Likewise, Theorem~\ref{thm:dual} can be
adapted (for instance, instead of $S_1-S'_1$ on the right, we need
$S_1-(S_2\cup S'_1)$).  Consistent with the hypotheses in
Theorem~\ref{thm:dirk}, Theorem~\ref{thm:inter} also applies provided
that $S_1\cap S_2$ is disjoint from $FI(M_i)$ and $FI(M_i^*)$ for
$i\in \{1,2\}$, Thus, an advantage of dealing with disjoint ground
sets is that it eliminates the need for assumptions about $FI(M_i)$
and $FI(M_i^*)$.

\vspace{5pt}

\begin{center}
  \textsc{Acknowledgements}
\end{center}

\vspace{3pt}

I am grateful to Jim Geelen for comments that led me to turn my
attention to the topic of intertwines and for valuable feedback on
this work.  I thank Dirk Vertigan for providing copies of the slides
of his talks on intertwines and spikes, and for permission to sketch
his construction.  I thank Anna de Mier for comments that improved the
exposition.

\end{document}